\documentclass[11pt]{article}
\usepackage[margin=1in]{geometry} 
\usepackage{amssymb,amsfonts,amsmath,bbm,mathrsfs,stmaryrd}
\usepackage{xcolor}
\usepackage{url}
\usepackage{ogonek}

\usepackage[colorlinks,
             linkcolor=black!75!red,
             citecolor=blue,
             pdftitle={Top_gen},
             pdfauthor={Alexandru Chirvasitu, Michael Brannan},
             pdfproducer={pdfLaTeX},
             pdfpagemode=None,
             bookmarksopen=true
             bookmarksnumbered=true]{hyperref}

\usepackage{tikz}
\usetikzlibrary{arrows,calc,decorations.pathreplacing,decorations.markings,intersections,shapes.geometric,through,fit,shapes.symbols,positioning,decorations.pathmorphing}

\usepackage{braket}

\usepackage[amsmath,thmmarks,hyperref]{ntheorem}
\usepackage{cleveref}

\crefname{section}{Section}{Sections}
\crefformat{section}{#2Section~#1#3} 
\Crefformat{section}{#2Section~#1#3} 

\crefname{subsection}{\S}{\S\S}
\crefformat{subsection}{#2\S#1#3} 
\Crefformat{subsection}{#2\S#1#3} 

\theoremstyle{plain}

\newtheorem{lemma}{Lemma}[section]
\newtheorem{proposition}[lemma]{Proposition}
\newtheorem{corollary}[lemma]{Corollary}
\newtheorem{theorem}[lemma]{Theorem}

\theoremstyle{nonumberplain}

\theoremstyle{plain}
\theorembodyfont{\upshape}
\theoremsymbol{\ensuremath{\blacklozenge}}

\newtheorem{definition}[lemma]{Definition}
\newtheorem{notation}[lemma]{Notation}

\newtheorem{remark}[lemma]{Remark}

\crefname{definition}{definition}{definitions}
\crefformat{definition}{#2definition~#1#3} 
\Crefformat{definition}{#2Definition~#1#3} 

\crefname{ex}{example}{examples}
\crefformat{example}{#2example~#1#3} 
\Crefformat{example}{#2Example~#1#3} 

\crefname{remark}{remark}{remarks}
\crefformat{remark}{#2remark~#1#3} 
\Crefformat{remark}{#2Remark~#1#3} 

\crefname{convention}{convention}{conventions}
\crefformat{convention}{#2convention~#1#3} 
\Crefformat{convention}{#2Convention~#1#3}

\crefname{lemma}{lemma}{lemmas}
\crefformat{lemma}{#2lemma~#1#3} 
\Crefformat{lemma}{#2Lemma~#1#3} 

\crefname{proposition}{proposition}{propositions}
\crefformat{proposition}{#2proposition~#1#3} 
\Crefformat{proposition}{#2Proposition~#1#3} 

\crefname{corollary}{corollary}{corollaries}
\crefformat{corollary}{#2corollary~#1#3} 
\Crefformat{corollary}{#2Corollary~#1#3} 

\crefname{theorem}{theorem}{theorems}
\crefformat{theorem}{#2theorem~#1#3} 
\Crefformat{theorem}{#2Theorem~#1#3} 

\crefname{assumption}{assumption}{Assumptions}
\crefformat{assumption}{#2assumption~#1#3} 
\Crefformat{assumption}{#2Assumption~#1#3} 

\crefname{equation}{}{}
\crefformat{equation}{(#2#1#3)} 
\Crefformat{equation}{(#2#1#3)}

\theoremstyle{nonumberplain}
\theoremsymbol{\ensuremath{\blacksquare}}

\newtheorem{proof}{Proof}

\newcommand\bC{{\mathbb C}}

\newcommand\cH{{\mathcal H}}

\newcommand\cT{{\mathcal T}}

\DeclareMathOperator{\id}{id}

\newcommand{\qedhere}{\mbox{}\hfill\ensuremath{\blacksquare}}

\title{Translations in quantum groups}
\author{Alexandru Chirvasitu}

\begin{document}

\date{}

\newcommand{\Addresses}{{% additional braces for segregating \footnotesize
  \bigskip
  \footnotesize

  \textsc{Department of Mathematics, University at Buffalo, Buffalo,
    NY 14260-2900, USA}\par\nopagebreak \textit{E-mail address}:
  \texttt{achirvas@buffalo.edu}

% %   \medskip
% %   
% %   \textsc{Department of Mathematics, institution,
% %     address}\par\nopagebreak \textit{E-mail address}:
% %   \texttt{??}
% % 
}}

\maketitle

\begin{abstract}
  Let $H$ be the Hopf $C^*$-algebra of continuous functions on a (locally) compact quantum group of either reduced or full type. We show that endomorphisms of $H$ that respect its right regular comodule structure are translations by elements of the largest classical subgroup of $G$.

  Furthermore, we show that for compact $G$ such an endomorphism is automatically an automorphism regardless of the quantum group norm on the $C^*$-algebra $H$.
\end{abstract}

\noindent {\em Key words: locally compact quantum group; translation; Peter-Weyl algebra}

\vspace{.5cm}

\noindent{MSC 2010: 20G42; 46L51; 22D25}

%\tableofcontents

%%%%%%%%%%%%%%%%%%%%%%%%%%%%%%%%%%%%%%%%%%%%%%%%%%%%%%%%%%%%%%%%%%%%%%%%%%%%%%%%%%%%%%%%%%%%%%%%%%%%%%%%%%%%%%%%%%
%%%%%%%%%%%%%%%%%%%%%%%%%%%%%%%%%%%%%%%%%%%%%%%%%%%%%%%%%%%%%%%%%%%%%%%%%%%%%%%%%%%%%%%%%%%%%%%%%%%%%%%%%%%%%%%%%%
\section*{Introduction}

The starting point for the present note is the observation that for a compact group $G$ the map sending $g\in G$ to the translation
\begin{equation*}
  C(G)\ni f\mapsto f(g -)\in C(G)
\end{equation*}
induces an isomorphism between the group $G$ and the monoid of $C^*$-algebra endomorphisms $\alpha:C(G)\to C(G)$ that respect the right comodule structure of $C(G)$ in the sense that
\begin{equation*}
    \begin{tikzpicture}[auto,baseline=(current  bounding  box.center)]
      \path[anchor=base] (0,0) node (hl) {$C(G)$} +(3,.5) node (hhu) {$C(G)\otimes C(G)$} +(3,-.5) node (hd) {$C(G)$} +(6,0) node (hhr) {$C(G)\otimes C(G)$};
      \draw[->] (hl) to[bend left=6] node[pos=.5,auto] {$\scriptstyle\Delta$} (hhu);
      \draw[->] (hl) to[bend right=6] node[pos=.5,auto,swap] {$\scriptstyle\alpha$} (hd);
      \draw[->] (hhu) to[bend left=6] node[pos=.5,auto] {$\scriptstyle \alpha \otimes\mathrm{id} $} (hhr);
      \draw[->] (hd) to[bend right=6] node[pos=.5,auto,swap] {$\scriptstyle\Delta$} (hhr);      
    \end{tikzpicture}  
\end{equation*}
commutes. 

The function algebra $H=C(G)$ of a compact {\it quantum} group is similarly equipped with a comultiplication, and one can ask for a similar description of all {\it translations} of the quantum group: those endomorphisms $\alpha:H\to H$ for which the analogous diagram is commutative.

More generally, the same problem can be posed in the context of {\it locally} compact quantum groups in the sense of \cite{KV,kv-vna}.

To make this more precise, let $G$ be a locally compact quantum group, $H=C_0(G)$ the reduced $C^*$-algebra of functions on $G$ vanishing at infinity and $H^u=C_0^u(G)$ its universal analogue. We write $G_{cl}$ for the largest classical quantum subgroup of $G$: this is nothing but the group of characters $H^u\to \bC$ equipped with the convolution product. An element $\chi\in G_{cl}$ induces an endomorphism $\alpha$ of $H^u$ in the category of non-unital $C^*$-algebras (i.e. a non-degenerate morphism $H^u\to M(H^u)$) defined by
\begin{equation}\label{eq:9}
    \begin{tikzpicture}[auto,baseline=(current  bounding  box.center)]
    \path[anchor=base] (0,0) node (h1)
      {$H^u$} +(2,.5) node (hh) {$M(H^u\otimes H^u)$} +(5,0) node (h2) {$M(H^u)$.};
      \draw[->] (h1) to[bend left=6] node[pos=.5,auto] {$\scriptstyle\Delta$} (hh);
      \draw[->] (hh) to[bend left=6] node[pos=.5,auto] {$\scriptstyle \chi\otimes\mathrm{id}$} (h2);
      \draw[->] (h1) to[bend right=6] node[pos=.5,auto,swap] {$\scriptstyle\alpha$} (h2);
  \end{tikzpicture}
\end{equation}
Note that in fact $\alpha$ takes values in $H^u$, and furthermore it respects the right $H^u$ coactions on its domain and codomain as discussed before. $\alpha$ is, in other words, a {\it translation} of $H^u$. One can now ask
\begin{itemize}
\item whether this correspondence is an isomorphism between $G_{cl}$ and the monoid of translations of $H^u$;
\item whether every such $\alpha$ descends to a translation of the reduced version $H$;
\item if so, then do we once more obtain an isomorphism between $G_{cl}$ and the monoid of translations of $H$?
\end{itemize}

That the classical answers survive in the quantum setting is one of the main result of the paper (an aggregate of \Cref{cor.full,pr.red,th.lcqg-univ,th.lcqg-red}):

\begin{theorem}\label{th.intro-main}
  Let $G$ be a locally compact quantum group and $C_0(G)$ and $C_0^u(G)$ its reduced and respectively full $C^*$-algebras of functions vanishing at infinity.

  The correspondence $\chi\mapsto \alpha$ described in \Cref{eq:9} gives isomorphisms between $G_{cl}$ and the monoids of translations of both $C^u_0(G)$ and $C_0(G)$.
\end{theorem}

Von Neumann algebra versions of this result appear in \cite{can} (for Kac algebras) and \cite{kn}. More specifically, \cite[Theorems 3.7, 3.9 and 3.11]{kn} jointly amount to the fact that von Neumann algebra morphisms $L^{\infty}(G)\to L^{\infty}(G)$ that preserve the right coaction are in bijection with $G_{cl}$ in the fashion described in \Cref{th.intro-main}.

Let us now specialize to the case when $G$ is compact, i.e. $C(G):=C_0(G)$ is unital. In general, one can equip the unique dense Hopf $*$-subalgebra of $C(G)$ with a number of different norms with that of $C(G)$ being minimal and that of $C^u(G)$ maximal (see \cite[discussion preceding Theorem 2.1]{bmt} for justification of the claim that the norm making the Haar state faithful is minimal).

The intermediate norms interpolating between $C(G)$ and $C^u(G)$ are often referred to as {\it exotic}; see e.g. \cite{ks,wrsm,r-wrsm,bg}. Although I do not know whether \Cref{th.intro-main} holds for exotic quantum group norms, it is nevertheless the case that every translation of a quantum group $C^*$-algebra is bijective (see \Cref{cor.bij}):

\begin{theorem}\label{th.intro-sec}
Every translation of the underlying $C^*$-algebra of a compact quantum group is an automorphism of said $C^*$-algebra.   
\end{theorem}

This answers a question posed by Piotr M. Hajac.

%%%%%%%%%%%%%%%%%%%%%%%%%%%%%%%%%%%%%%%%%%%%%%%%%%%%%%%%%%%%%%%%%%%%%%%%%%%%%%%%%%%%%%%%%%%%%%%%%%%%%%%%%%%%%%%%%%
\subsection*{Acknowledgements}

This work is supported in part by NSF grant DMS-1801011.

I am grateful for enlightening input from Ludwik D\k{a}browski, Piotr Hajac, Pawe{\l} Kasprzak and Mariusz Tobolski on earlier drafts.

Additionally, the suggestions of the anonymous referee have improved the content significantly, pointing out the references \cite{can,kn} and much more. 

%%%%%%%%%%%%%%%%%%%%%%%%%%%%%%%%%%%%%%%%%%%%%%%%%%%%%%%%%%%%%%%%%%%%%%%%%%%%%%%%%%%%%%%%%%%%%%%%%%%%%%%%%%%%%%%%%%
%%%%%%%%%%%%%%%%%%%%%%%%%%%%%%%%%%%%%%%%%%%%%%%%%%%%%%%%%%%%%%%%%%%%%%%%%%%%%%%%%%%%%%%%%%%%%%%%%%%%%%%%%%%%%%%%%%
\section{Preliminaries}\label{se.prel}

%%%%%%%%%%%%%%%%%%%%%%%%%%%%%%%%%%%%%%%%%%%%%%%%%%%%%%%%%%%%%%%%%%%%%%%%%%%%%%%%%%%%%%%%%%%%%%%%%%%%%%%%%%%%%%%%%%
\subsection{Locally compact quantum groups}\label{subse.prel-lcqg}

We will work with {\it locally} compact quantum groups (LCQGs for short) of either reduced or full type, in the sense of \cite{KV} and respectively \cite{univ}. For locally compact quantum group morphisms we refer to \cite{SLW12}; see also the preliminary section of \cite{DKSS} for a quick review of the theory.

The category of non-unital $C^*$-algebras we work with is the usual one, where morphisms $A\to B$ are non-degenerate homomorphisms $A\to M(B)$ to the {\it multiplier algebra} $M(B)$ of $B$ (this is the space denoted by $\mathrm{Mor}(A,B)$ in \cite[$\S$1.1]{DKSS}). Throughout the present subsection $H=C_0(G)$ will be the underlying non-unital $C^*$-algebra of a reduced LCQG $G$, and $H_u$ the associated $C^*$-algebra (whose dual object we denote, as before, by $G_u$).

\cite[Section 1]{KV} contains all of the $C^*$ weight theory we will need (much more in fact). For a weight $\varphi$ on a (possibly non-unital) $C^*$-algebra $A$ we denote
\begin{align*}
  M^+_{\varphi} &:= \{x\in A^+\ |\ \varphi(x)<\infty\},\ M_{\varphi}:=\mathrm{span}\; M^+_{\varphi},\\
  N_{\varphi}&:=\{x\in A\ |\ \varphi(x^*x)<\infty\}.
\end{align*}

We write $C_0(G)$ and $C_0^u(G)$ for the reduced and full function $C^*$-algebras attached to the LCQG $G$. As has become customary in the literature, denote by $\varphi$ the left-invariant Haar weight. It is {\it proper} (i.e. non-zero, lower semicontinuous and densely defined), faithful on $C_0(G)$, and left invariance means that
\begin{equation}\label{eq:2}
  \varphi((\omega\otimes \id)\Delta x) = \omega(1)\varphi(x),\ \forall x\in M^+_{\varphi},\ \forall \omega\in C_0(G)_+^*
\end{equation}
(see \cite[Definition 2.2]{KV}). Recall (e.g. \cite[$\S$1.2]{KV}) that proper weights admit unique extensions to the multiplier algebras of their domains; we take such extensions for granted.%, and decorate the multiplier algebra versions of $M^+_{\varphi}$, $M_{\varphi}$ and $N_{\varphi}$ with overlines: $\overline{M}^+_{\varphi}$, $\overline{M}_{\varphi}$ and $\overline{N}_{\varphi}$ respectively. 

We have a GNS representation $(\cH_{\varphi},\pi_{\varphi},\Lambda)$ attached to $\varphi$ as in \cite[Definition 1.2]{KV}: $\cH_{\varphi}$ is a Hilbert space and
\begin{align*}
  \pi_{\varphi}&:C^u_0(G)\to C_0(G)\to B(\cH_{\varphi})\\
  \Lambda&:N_{\varphi}\to \cH_{\varphi}
\end{align*}
satisfying $\Lambda(xy)=\pi_{\varphi}(x)\Lambda(y)$ for all relevant $x$ and $y$. We regard $C_0(G)$ as a subalgebra of $B(\cH_{\varphi})$ via $\pi_{\varphi}$, and $L^{\infty}(G)$ is its von Neumann closure therein.

The {\it multiplicative unitary} $W\in B(\cH_{\varphi}^{\otimes 2})$ is defined by
\begin{equation*}
  W^*(\Lambda(x)\otimes \Lambda(y)) = (\Lambda\otimes\Lambda)(\Delta(y)(x\otimes 1)),\ \forall x,y\in N_{\varphi}. 
\end{equation*}
Then $W$ induces the comultiplication on $L^{\infty}(G)$ via
\begin{equation*}
  \Delta(x) = W^*(1\otimes x)W,
\end{equation*}
and $W$ is contained in $M(C_0(\widehat{G})\otimes C_0(G))$ for the {\it Pontryagin dual} $\widehat{G}$ of $G$: by definition, $C_0(\widehat{G})$ is the norm-closure of
\begin{equation*}
\{(\omega\otimes\id)W\ |\ \omega\in B(\cH_{\varphi})^*\}
\end{equation*}
(\cite[Definition 8.1]{KV}). We have
\begin{align*}
  (\Delta\otimes\id)W &= W_{13}W_{23}\\
  (\id\otimes\Delta)W&=W_{13}W_{12},
\end{align*}
relations which we will use below (note that the right hand comultiplication in the second line of the display above is that of $C_0(\widehat{G})$).

%%%%%%%%%%%%%%%%%%%%%%%%%%%%%%%%%%%%%%%%%%%%%%%%%%%%%%%%%%%%%%%%%%%%%%%%%%%%%%%%%%%%%%%%%%%%%%%%%%%%%%%%%%%%%%%%%%
\subsection{Compact quantum groups}\label{subse.prel-cqg}

Although these are technically examples of LCQGs (i.e. those for which $C_0(G)$ is unital), there are several aspects peculiar to the compact case that we outline here very briefly, referring to \cite{Wor98,KusTus99} for further background. 

The underlying Hopf $C^*$-algebra $H=C(G)$ has a unique dense Hopf $*$-algebra we denote by $\mathrm{Pol}(G)$ (the {\it Peter-Weyl algebra} of $G$). It is cosemisimple as a coalgebra and hence breaks up as a direct sum of matrix coalgebras, one for each irreducible right $H$-comodule.

\section{Translations}\label{se.tr}

Let $H=C(G)$ be the Hopf $C^*$-algebra underlying a CQG $G$.

\begin{definition}\label{def.tr}
  A {\it translation} of $G$ or $H$ is a $C^*$-algebra morphism
  $\alpha$ making the diagram
  \begin{equation}\label{eq:tr}
    \begin{tikzpicture}[auto,baseline=(current  bounding  box.center)]
      \path[anchor=base] (0,0) node (hl)
      {$H$} +(2,.5) node (hhu) {$H\otimes H$} +(2,-.5) node (hd) {$H$} +(4,0) node (hhr) {$H\otimes H$};
      \draw[->] (hl) to[bend left=6] node[pos=.5,auto] {$\scriptstyle\Delta$} (hhu);
      \draw[->] (hl) to[bend right=6] node[pos=.5,auto,swap] {$\scriptstyle\alpha$} (hd);
      \draw[->] (hhu) to[bend left=6] node[pos=.5,auto] {$\scriptstyle \alpha \otimes\mathrm{id} $} (hhr);
      \draw[->] (hd) to[bend right=6] node[pos=.5,auto,swap] {$\scriptstyle\Delta$} (hhr);      
    \end{tikzpicture}
  \end{equation}
  commutative.

  We denote by $\cT(H)$ or $\cT(G)$ the monoid of translations of $G$ equipped with {\it opposite} composition.
\end{definition}

\begin{remark}
  The convention that we compose translations backwards obviates the need to reverse the multiplication on $G_{cl}$ in \Cref{th.intro-main}.   
\end{remark}

We topologize $\cT(H)$ for $H=C(G)$ in the standard way by the topology of pointwise convergence in the norm of $H$.

Now let $\alpha\in \cT(G)$. For each $\gamma\in \mathrm{Irr}(G)$ the matrix coalgebra $H_\gamma\subset H$ is preserved by $\alpha$. If $d_\gamma$ is the dimension of $\gamma$ and $u_\gamma=(u_{ij})\in M_{d_\gamma}(H)$ is a unitary matrix whose entries form a basis for $H_\gamma$, then the action of $\alpha$ on $H_\gamma$ is implemented by multiplying $u_\gamma$ by a unitary $d_\gamma\times d_\gamma$ matrix $T_\gamma=T^\alpha_\gamma$.

Since $\varphi\in \cT(G)$ is uniquely determined by its action on the dense Hopf $*$-subalgebra $\mathrm{Pol}(G)$, we thus have an {\it embedding}
  \begin{equation}\label{eq:1}
    \cT(G)\subseteq U_H:=\prod_{\gamma\in \mathrm{Irr}(G)}U_{d_\gamma} 
  \end{equation}
  sending $\alpha$ to the tuple $(T^\alpha_\gamma)_\gamma$.

Our first observation is

\begin{proposition}\label{pr.cpct}
The embedding \Cref{eq:1} realizes $\cT(G)$ as a closed sub-monoid of the right hand side. 
\end{proposition}
\begin{proof}
  We have already noted the injectivity of the map \Cref{eq:1}, and by its very definition it intertwines the composition of translations and the group operation on its codomain $U_H$. It thus remains to argue that the map is continuous and its image is closed.
  
% %   To check the continuity of \Cref{eq:1}, consider a net $(\alpha_{\eta})_{\eta}$ converging to $\alpha$ in he standard topology. This means that
% %   \begin{equation*}
% %     \alpha_{\eta}(x)\to \alpha(x)
% %   \end{equation*}
% %   for all $x\in H$. In particular this holds for every
% %   \begin{equation*}
% %    x\in \mathrm{Pol}(G) = \bigoplus_{\gamma} H_{\gamma}\subset H, 
% %   \end{equation*}
% %   meaning that we have
% %   \begin{equation*}
% %     T^{\alpha_{\eta}}_{\gamma}\to T^{\alpha}_{\gamma} 
% %   \end{equation*}
% %   for all $\gamma$. This, however, is precisely what convergence of $(T^{\alpha_{\eta}}_{\gamma})_{\gamma}$ to $(T^{\alpha}_{\gamma})_{\gamma}$ in $\prod_{\gamma}U_{d_{\gamma}}$ entails.
% %
  
Concerning the closure of $\cT(G)$ in $U_H$ via \Cref{eq:1} consider a net $(\alpha_{\eta})_{\eta}\subset \cT(G)$ converging to $\alpha\in U_H$. The latter can be identified with a self-map of
  \begin{equation*}
    \mathrm{Pol}(G) = \bigoplus_{\gamma} H_{\gamma}
  \end{equation*}
  as before, acting as multiplication by its $\gamma$-component unitary matrix on $H_{\gamma}$. This is a right comodule endomorphism of $\mathrm{Pol}(G)$. It is also an algebra morphism: multiplicativity is a condition on arbitrary pairs of elements, and hence a closed condition under pointwise convergence.

  It remains to argue that $\alpha$ must lift to an endomorphism of the $C^*$-algebra $H$ assuming that $\alpha_{\eta}$ are such endomorphisms. Equivalently, since $\mathrm{Pol}(G)$ is dense in $H$, it is enough to show that $\alpha$ is contractive on $\mathrm{Pol}(G)$ (with respect to the norm of $H$). 

  To see this, fix $x\in \mathrm{Pol}(G)$. $x$ is contained in some finite-dimensional subspace
  \begin{equation*}
    V=\bigoplus_{\gamma\in F}H_{\gamma} 
  \end{equation*}
  for a finite subset $F\subset \mathrm{Irr}(G)$. $V$ is preserved by $\alpha_{\eta}$ as well as $\alpha$, and hence
  \begin{equation*}
    \alpha_{\eta}(x)\to \alpha(x)
  \end{equation*}
in the unique separated vector space topology on $V$. In particular this is also the topology inherited from the norm $\|-\|$ on $H$, and hence 
  \begin{equation*}
    \|\alpha(x)\| = \lim_{\eta}\|\alpha_{\eta}(x)\|\le \|x\|. 
  \end{equation*}
  because $\alpha_{\eta}$ all lift to endomorphisms of the $C^*$-algebra $H$ and are thus contractive. This finishes the proof.
\end{proof}

As an immediate consequence of \Cref{pr.cpct} we have

\begin{corollary}\label{cor.bij}
  Every translation of $H=C(G)$ is an automorphism of the $C^*$-algebra $H$. 
\end{corollary}
\begin{proof}
  We know from \Cref{pr.cpct} that $\cT=\cT(H)$ is a closed sub-monoid of a compact group $U_H$. Being a sub-monoid of a group $\cT$ is {\it cancellative} (i.e. $ab=ab'\Rightarrow b=b'$ and similarly for $ba=b'a$). This means that it is, in fact, a closed sub{\it group}: see e.g. \cite[Lemma B.1]{rs} for (a strengthening of) the well known result to the effect that compact cancellative semigroups are groups.
\end{proof}

%%%%%%%%%%%%%%%%%%%%%%%%%%%%%%%%%%%%%%%%%%%%%%%%%%%%%%%%%%%%%%%%%%%%%%%%%%%%%%%%%%
\subsection{Classifying translations}

A simple observation: recall that for any algebra $A$, every (left, say) $A$-module endomorphism of $A$ is automatically of the form
\begin{equation*}
  A\ni x\mapsto xa
\end{equation*}
for some $a\in A$. In other words, $A$ is its own algebra of $A$-module endomorphisms.

Dually, for a coalgebra $C$ over a field $k$, every right $C$-comodule endomap $\alpha$ of $C$ will be of the form
\begin{equation*}
  c\mapsto \chi(c_1)c_2,
\end{equation*}
where $c\mapsto c_1\otimes c_2$ is Sweedler notation for comultiplication and $\chi:C\to k$ is a linear map: one simply sets $\chi=\varepsilon\circ \alpha$.

If furthermore $C=H$ is a Hopf algebra (or even just a bialgebra) and $\alpha$ is also a translation in the sense of \Cref{def.tr}, then $\chi:H\to k$ is easily seen to be an algebra map; it is, after all, just the composition
\begin{equation}\label{eq:pa}
  \begin{tikzpicture}[auto,baseline=(current  bounding  box.center)]
    \path[anchor=base] (0,0) node (h1)
      {$H$} +(2,.5) node (hh) {$H\otimes H$} +(4,0) node (h2) {$H$} +(6,0) node (k) {$k$};
      \draw[->] (h1) to[bend left=6] node[pos=.5,auto] {$\scriptstyle\Delta$} (hh);
      \draw[->] (hh) to[bend left=6] node[pos=.5,auto] {$\scriptstyle \chi\otimes\mathrm{id}$} (h2);
      \draw[->] (h1) to[bend right=6] node[pos=.5,auto,swap] {$\scriptstyle\alpha$} (h2);
      \draw[->] (h2) to node[pos=.5,auto] {$\scriptstyle\varepsilon$} (k);
      \draw[->] (h1) to[bend right=20] node[pos=.5,auto,swap] {$\scriptstyle\chi$} (k);      
  \end{tikzpicture}
\end{equation}
of $\alpha$ and the counit $\varepsilon$, both of which are algebra morphisms. Finally, $*$-structures come along for the ride for Hopf $*$-algebras (i.e. if $\alpha$ is a $*$-algebra endomorphism then so is $\chi:H\to\bC$, and conversely).

\begin{notation}\label{nt.pa}
  In order to indicate the dependence of $\alpha$ and $\chi$ on each other in \Cref{eq:pa}, we will sometimes denote them by $\alpha_\chi$ and $\chi_\alpha$ respectively.
\end{notation}

Now, if $H=C(G)$ is a Hopf $C^*$-algebra as we have been assuming and $\alpha\in\cT(G)$ is a translation, then the preceding discussion applies to the dense Hopf $*$-subalgebra $\mathrm{Pol}(H)$ (which is automatically preserved by $\alpha$). This means that the restriction of $\alpha$ to $\mathrm{Pol}(H)$ is of the form $\alpha_\chi$ (see \Cref{nt.pa}) for a unique character $\chi$ of the $*$-algebra $\mathrm{Pol}(H)$, or equivalently, of the full Hopf $C^*$-algebra $H_u$ with $\mathrm{Pol}(H_u)=\mathrm{Pol}(H)$.

\begin{notation}\label{nt.cl}
  For a CQG $H=C(G)$ we denote by $\mathrm{pt}(G)$ (for `points of $G$') the space of characters $H\to \bC$, or equivalently, the space of characters of the abelianization $H\to H_{cl}$ of $H$.
\end{notation}

To summarize:

\begin{proposition}\label{pr.t-in-char}
  Let $H=C(G)$ for a CQG $G$. Then, the map
  \begin{equation*}
    \cT(G)\ni \alpha\mapsto \chi_\alpha
  \end{equation*}
  is an embedding of $\cT(G)$ as a closed subgroup of the compact group $G_{cl}$.
  \qedhere
\end{proposition}

Conversely, for every character $\chi:H\to \bC$ we can define the translation $\alpha_\chi$ via the left hand half of \Cref{eq:pa}. We thus obtain the following version of \Cref{pr.cpct}.

\begin{proposition}\label{pr.char-in-t}
  The map
  \begin{equation*}
    \mathrm{pt}(G)\ni \chi\mapsto \alpha_\chi
  \end{equation*}
  is an embedding of $\mathrm{pt}(G)$ as a closed sub-semigroup of the
  compact group $\cT(G)$.  \qedhere
\end{proposition}

Moreover, since for full CQGs we have $G=G_u$ and hence the sandwich
\begin{equation*}
  \mathrm{pt}(G)\subseteq \cT(G)\subseteq G_{cl}
\end{equation*}
provided by \Cref{pr.t-in-char,pr.char-in-t} collapses, we obtain

\begin{theorem}\label{cor.full}
  If $G$ is a full compact quantum group, $\cT(G)$ is isomorphic to the maximal classical subgroup $\mathrm{pt}(G)\subseteq G$. \qedhere
\end{theorem}

We have an analogous result for {\it reduced} compact quantum groups:

\begin{theorem}\label{pr.red}
  If $G$ is a reduced compact quantum group then the embedding $\cT(G)\to G_{cl}$ from \Cref{pr.t-in-char} is an isomorphism.
\end{theorem}
\begin{proof}
  Let $H=C(G)$. We have to argue that every element $\chi\in G_{cl}$, acting as an automorphism of the CQG algebra $\mathrm{Pol}(H)$, lifts to an automorphism of the reduced $C^*$ closure $H$ of $\mathrm{Pol}(H)$. This follows from the fact that $H$ is the GNS closure with respect to the Haar state $h$ on $\mathrm{Pol}(H)$, and $h$ is preserved by $\chi$.
\end{proof}

%%%%%%%%%%%%%%%%%%%%%%%%%%%%%%%%%%%%%%%%%%%%%%%%%%%%%%%%%%%%%%%%%%%%%%%%%%%%%%%%%%
\subsection{The locally compact case}\label{subse.lcqg}

We write $H=C_0(G)$ and $H^u=C_0^u(G)$. The present subsection is the reason why we have chosen to work with {\it left} comodule morphisms $\alpha$: it allows better agreement with the literature on LCQGs that we reference here. 

We can define the group $\cT(G_u)=\cT(H^u)$ of translations of $G_u$ as in \Cref{def.tr}; it consists of non-degenerate endomorphisms $\alpha$ of $H^u$ that preserve the right regular coaction of $H^u$, in the sense that the diagram
  \begin{equation}\label{eq:8}
    \begin{tikzpicture}[auto,baseline=(current  bounding  box.center)]
      \path[anchor=base] (0,0) node (hl)
      {$H^u$} +(2,.5) node (hhu) {$M(H^u\otimes H^u)$} +(2,-.5) node (hd) {$M(H^u)$} +(5,0) node (hhr) {$M(H^u\otimes H^u)$};
      \draw[->] (hl) to[bend left=6] node[pos=.5,auto] {$\scriptstyle\Delta$} (hhu);
      \draw[->] (hl) to[bend right=6] node[pos=.5,auto,swap] {$\scriptstyle\alpha$} (hd);
      \draw[->] (hhu) to[bend left=6] node[pos=.5,auto] {$\scriptstyle \alpha \otimes\mathrm{id}$} (hhr);
      \draw[->] (hd) to[bend right=6] node[pos=.5,auto,swap] {$\scriptstyle\Delta$} (hhr);      
    \end{tikzpicture}
  \end{equation}
analogous to \Cref{eq:tr} commutes.

We again equip $\cT(G_u)$ with its {\it standard topology}: as before, this is the pointwise-norm topology when regarded as a space of $C^*$-algebra morphisms $H^u\to M(H^u)$.

\begin{remark}
  It is easy to see that with its standard topology, $\cT(G_u)$ is a locally compact group.
\end{remark}

As in the discussion at the beginning of \Cref{se.tr}, we denote by $G_{cl}$ the maximal classical closed subgroup of $G_u$; it is the spectrum of the abelianization $(H^u)_{cl}$, i.e. the space of characters of $H^u$. The version of \Cref{eq:pa} relevant to universal locally compact quantum groups is
\begin{equation}\label{eq:pa-lc}
  \begin{tikzpicture}[auto,baseline=(current  bounding  box.center)]
    \path[anchor=base] (0,0) node (h1)
      {$H^u$} +(2,.5) node (hh) {$M(H^u\otimes H^u)$} +(5,0) node (h2) {$M(H^u)$} +(7,0) node (k) {$k$,};
      \draw[->] (h1) to[bend left=6] node[pos=.5,auto] {$\scriptstyle\Delta$} (hh);
      \draw[->] (hh) to[bend left=6] node[pos=.5,auto] {$\scriptstyle \chi\otimes\mathrm{id}$} (h2);
      \draw[->] (h1) to[bend right=6] node[pos=.5,auto,swap] {$\scriptstyle\alpha$} (h2);
      \draw[->] (h2) to node[pos=.5,auto] {$\scriptstyle\varepsilon$} (k);
      \draw[->] (h1) to[bend right=20] node[pos=.5,auto,swap] {$\scriptstyle\chi$} (k);      
  \end{tikzpicture}
\end{equation}
and defines back-and-forth mutually inverse maps
\begin{equation}\label{eq:bf}
  \begin{tikzpicture}[auto,baseline=(current  bounding  box.center)]
    \path[anchor=base] (0,0) node (t)
      {$\cT(G_u)$} +(4,0) node (cl) {$G_{cl}$,};
      \draw[->] (t) to[bend left=10] node[pos=.5,auto] {$\alpha\mapsto \chi_\alpha$} (cl);
      \draw[->] (cl) to[bend left=10] node[pos=.5,auto] {$\alpha_\chi\mapsfrom \chi$} (t);      
  \end{tikzpicture}
\end{equation}
where the fact that $\chi$ is not the zero map follows from the non-degeneracy of $\alpha$. We thus arrive at

\begin{theorem}\label{th.lcqg-univ}
  The rightward map in \Cref{eq:bf} defines an isomorphism
  \begin{equation*}
    \cT(G_u)\cong G_{cl}.
  \end{equation*}
\end{theorem}
\begin{proof}
  We already have mutually inverse bijections preserving the group structures, so the only claim still to be verified is that these bijections are both continuous.

  It is immediate from the definition of the pointwise-norm topology that for a fixed morphism $f$ of (non-unital) $C^*$-algebras the map $g\mapsto g\circ f$ is continuous. Since both $\alpha\mapsto \chi_{\alpha}$ and $\chi\mapsto \alpha_{\chi}$ are of this form, this identifies $\cT(G_u)$ topologically with the space of characters $H^u\to \bC$ equipped with the pointwise-norm topology.

  Since characters $H^u\to \bC$ coincide with characters of the abelianization $C_0(G_{cl})$ of $H^u$ again topologized via the pointwise-norm topology, the claim amounts to the fact that the locally compact topology of $G_{cl}$ can be recovered as the pointwise-norm topology on
  \begin{equation*}
    \text{characters } C_0(G_{cl})\to \bC. 
  \end{equation*}
  This, however, is nothing but the Gelfand-Naimark theorem.
\end{proof}

We can now define $\cT(H)=\cT(G)$ as in the universal case. The analogue, in this case, of \Cref{cor.bij} is \Cref{th.lcqg-red} below. It is analogous to \cite[Theorems 3.7, 3.9 and 3.11]{kn}, being a $C^*$ (as opposed to von Neumann) variant of that material. In fact, we will apply the results of \cite{kn} directly once we transport the problem to GNS von Neumann algebras.

\begin{theorem}\label{th.lcqg-red}
  Every translation of $H^u=C^u_0(G)$ descends to one of $H=C_0(G)$, and this correspondence induces an isomorphism
  \begin{equation*}
    \cT(G)\cong G_{cl}
  \end{equation*}
\end{theorem}
\begin{proof}
  We already know from \Cref{th.lcqg-univ} that translations of $H^u$ are actual left translations by elements in the classical locally compact group $G_{cl}$. These preserve the left-invariant Haar weight $\varphi$ on $H^u$ and hence descend to the quotient $H$ of $H^u$ on which $\varphi$ is faithful.

  It now remains to show that {\it every} translation of $H$ arises in this fashion, as a left translation by some element in $G_{cl}$. We fix a translation $\alpha\in \cT(H)$ throughout the rest of the proof.

  {\bf Claim: $\alpha$ lifts to a normal self-map $L^{\infty}(G)$.}

  The goal here is to show that $\alpha$ sends $M^+_{\varphi}$ to $\overline{M}^+_{\varphi}$ and $\varphi\circ\alpha = \varphi$. To that end, recall that \Cref{eq:2} is valid for all positive elements in $H$ (or indeed even $M(H)$). For that reason, applying $\omega\otimes\varphi$ to the image of $x\in M^+_{\varphi}$ through \Cref{eq:8} we obtain
\begin{equation*}
  \omega(1)\varphi(\alpha(x)) = \varphi((\omega\otimes \id)\Delta\alpha(x)) = \varphi((\omega\circ\alpha\otimes\id) \Delta x)
\end{equation*}
for all positive $x\in M^+_{\varphi}$ and positive functionals $\omega\in H^*_+$, where the second equality uses the commutativity of \Cref{eq:8}. In turn, this equals
\begin{equation*}
  \omega(\alpha(1))\varphi(x) = \omega(1)\varphi(x). 
\end{equation*}
This does indeed prove the desired conclusion $\varphi\circ\alpha = \varphi$ and hence Claim 1 because then $\alpha$ lifts to the closure $L^{\infty}(G)$ of $H$ in the GNS representation of the $\alpha$-invariant weight $\varphi$.

As explained before, once we have lifted $\alpha$ to $L^{\infty}$ we can conclude by \cite[Theorems 3.7, 3.9 and 3.11]{kn}, with 3.7 and 3.11 applied to $G$ and 3.9 applied to the dual $\widehat{G}$. 
\end{proof}

% % old proof
% % 
% %  We now have a morphism $\alpha:L^{\infty}(G)\to L^{\infty}(G)$ which in the language of \cite{ss2} is a {\it left multiplier} of the latter von Neumann algebra, i.e. preserves the right comodule structure induced by the comultiplication
% %  \begin{equation*}
% %    \Delta: L^{\infty}(G) \to L^{\infty}(G)\otimes L^{\infty}(G). 
% %  \end{equation*}
% % 
% %  Now, according to \cite[Theorem 6]{ss2} there is an element $g\in M(C_0(\widehat{G}))$ such that
% %  \begin{equation}\label{eq:6}
% %    (\alpha\otimes \id)W = (1\otimes g) W. 
% %  \end{equation}
% %  Now apply $\id\otimes \Delta$ to this equation. Using $(\id\otimes \Delta)W=W_{13}W_{12}$ and the fact that $\alpha$ is multiplicative we obtain $\Delta g = g\otimes g$. It follows that $g$ is contained in the largest cocommutative von Neumann Hopf $*$-subalgebra of $L^{\infty}(\widehat{G})$, i.e. the group von Neumann algebra of $G_{cl}$. 
% % 
% %  By \cite[Theorem VII.3.9]{tak2} the equation $\Delta g = g\otimes g$ for an element of a group von Neumann algebra means that $g$ is an element of the group: $g\in G_{cl}$. Applying $\id\otimes\omega$ to \Cref{eq:6} for arbitrary $\omega\in C_0(\widehat{G})^*$ we see that indeed $\alpha$ operates on
% %  \begin{equation}\label{eq:7}
% %    (\id\otimes \omega)W\in M(C_0(G)) 
% %  \end{equation}
% %  as left translation by $g$. Since the elements \Cref{eq:7} span a dense subset of the latter algebra, the conclusion follows.
% % 
% % end of old proof 

As a consequence, we have the analogue of \Cref{cor.bij} in the locally compact setting, albeit only in the universal and reduced cases:

\begin{corollary}\label{th.lcqg-iso}
  Every translation of a reduced or universal locally compact quantum group is an isomorphism of the respective $C^*$-algebra $C_0(G)$ or $C_0^u(G)$. 
\end{corollary}
\begin{proof}
  It is immediate from the characterization of translations given in \Cref{th.lcqg-univ,th.lcqg-red}: they are all translations by elements of the classical group $G_{cl}$, and each such map admits an inverse (namely translation by the inverse of said element).
\end{proof}

\begin{remark}
  There is a more general approach to locally compact quantum groups, based around the notion of a {\it manageable multiplicative unitary} (see e.g. \cite{wor-man,wor-mu,wor-mu2}).

  \Cref{th.lcqg-red} holds in the more general setting as well (i.e. translations are in bijection with the elements of the maximal classical subgroup) because \cite[Proposition 3.2]{dws-pos}, which in turn \cite[Theorem 6]{ss2} paraphrases, holds in the more general setup.
\end{remark}

\bibliographystyle{plain}
\addcontentsline{toc}{section}{References}

\def\polhk#1{\setbox0=\hbox{#1}{\ooalign{\hidewidth
  \lower1.5ex\hbox{`}\hidewidth\crcr\unhbox0}}}
  \def\polhk#1{\setbox0=\hbox{#1}{\ooalign{\hidewidth
  \lower1.5ex\hbox{`}\hidewidth\crcr\unhbox0}}}

\Addresses

\end{document}